\documentclass[12pt]{amsart}
\usepackage{amscd,amssymb,epsfig,mathtools}
\calclayout

\newif\ifpdf

\ifx\pdfoutput\undefined

  \pdffalse 
\else 
  \pdftrue\pdfoutput=1 
\fi 
 
\numberwithin{equation}{section}

\newcommand{\R}{\mathbb{R}}

\renewcommand{\div}{\operatorname{div}}

\newcommand{\tir}[1]{\ensuremath{\overline {#1}}} 
\newtheorem{thm}{Theorem}[section] 
 
\newtheorem{lemma}[thm]{Lemma}

\newtheorem{defn}[thm]{Definition} 
 
\newtheorem{rem}[thm]{Remark}

\def\whsq{\vbox to 5.8pt 
{\offinterlineskip\hrule 
\hbox to 5.8pt{\vrule height 
5.1pt\hss\vrule height 5.1pt}\hrule}}

\def\<{\langle} 
\def\>{\rangle} 
\def\PP{{\mathop{{\rm I}\kern-.2em{\rm P}}\nolimits}} 
\def\FF{{\mathop{{\rm I}\kern-.2em{\rm F}}\nolimits}}   
\def\ZZ{{\mathop{{\rm I}\kern-.2em{\rm Z}}\nolimits}} 
 
\newlength{\sidemargin} 
\begin{document}
\title[]{
Iterative methods for $k$-Hessian equations
}

\author{Gerard Awanou}
\address{Department of Mathematics, Statistics, and Computer Science, M/C 249. University of Illinois at Chicago, Chicago, IL 60607-7045, USA}

\maketitle

\begin{abstract}
On a domain of the $n$-dimensional Euclidean space, and for an integer $k=1,\ldots,n$, the $k$-Hessian equations are fully nonlinear elliptic equations for $k >1$ and consist of the Poisson equation for $k=1$ and the Monge-Amp\`ere equation for $k=n$. We analyze for smooth non degenerate solutions a 9-point finite difference scheme. We prove that the discrete scheme has a locally unique solution with a quadratic convergence rate.
In addition we propose new iterative methods which are numerically shown to work for non smooth solutions. A connection of the latter with a popular Gauss-Seidel method for the Monge-Amp\`ere equation is established and new Gauss-Seidel type iterative methods for $2$-Hessian equations are introduced. 

\end{abstract}

\section{Introduction}

Let $\Omega$ be a bounded, connected open subset of $\R^n, n\geq 2$ with boundary denoted $\partial \Omega$. 
Let $u \in C^2(\Omega)$ and for $x \in \Omega$, let $D^2 u(x)=\bigg( (\partial^2 u(x))(\partial x_i \partial x_j)\bigg)_{i,j=1,\ldots, n}  $ 
denote its Hessian. We denote the eigenvalues of $D^2 u(x)$ by $\lambda_i(x), i=1,\ldots, n$. 
For $1 \leq k \leq n$, the $k$-Hessian operator is defined as
\begin{align*}
S_k(D^2 u) = \sum_{i_1 < \cdots < i_k} \lambda_{i_1}  \cdots \lambda_{i_k}.
\end{align*}
 We note that $S_1(D^2 u) = \Delta u$ is the Laplacian operator and $S_n(D^2 u) = \det D^2 u$ is the Monge-Amp\`ere operator. For $k \geq 2$, we are interested in the numerical approximation of 
 solutions of the Dirichlet problem for the $k$-Hessian equation 
\begin{equation}
S_k(D^2 u) = f \, \text{in} \, \Omega, u=g  \, \text{on} \, \partial \Omega, \label{k-H1}
\end{equation}
with $f$ and $g$ given and $f \geq 0$. 

\subsection{Local existence, uniqueness and quadratic convergence rate for a finite difference discretization}
 Let $u^0$ be a sufficiently close initial guess to the smooth solution $u$ of \eqref{k-H1}. Consider the iterative method
 \begin{align} \label{broyden}
\begin{split}
\div \bigg( \{S_k^{ij}(D^2 u^0) \} D u^{m+1} \bigg)& = \div \bigg( \{S_k^{ij}(D^2 u^0) \} D u^{m} \bigg) +f-S_k (D^2 u^m) \, \text{in} \,  \Omega  \\
 u^{m+1} & = g \, \text{on} \, \partial \Omega, 
 \end{split}
\end{align}
where $ \{S_k^{ij}(D^2 u^0) \}$ is a matrix which generalizes the cofactor matrix of $D^2 u^0$. 

We prove the convergence of \eqref{broyden} at the continuous level in H$\ddot{\text{o}}$lder spaces. A discrete version of \eqref{broyden} is also shown to converge to a solution of a 9-point stencil discretization of \eqref{k-H1}. This establishes the local existence and uniqueness of a discrete solution. In addition the convergence rate of the discretization is shown to be quadratic.


It is reasonable to expect that the discrete version of the iterative method \eqref{broyden} will retrieve the correct solution when it is smooth and non degenerate. As with Newton's method it is not effective for non smooth  and degenerate solutions. For these, we advocate iterative methods like the subharmonicity preserving iterations described below. The discrete version of \eqref{broyden} is used in this paper to prove the local solvability of the 9-point scheme when $u$ is smooth and non degenerate. These results form a building block of a theory which explains why standard discretizations work for non smooth solutions \cite{Awanou-Std-fd-jsc}. In addition results for smooth solutions are also needed for the analysis of hybrid schemes where the 9 point scheme is used in part of the region where the solution is smooth and a monotone scheme elsewhere \cite{AwanouHybrid}.

\subsection{Newton's method}
If one is only interested in smooth solutions, Newton's method is the most appropriate method. 
We analyze the convergence of Newton's method for solving \eqref{k-H1} when it has a smooth solution.

\subsection{Numerical work for subharmonicity preserving iterations}

A smooth function $u$ is said to be $k$-convex if $S_l (D^2 u) \geq 0, 1 \leq l \leq k$. Convexity of a function can be shown to be equivalent to $n$-convexity, Lemma \ref{n-convexity}. 
It is of interest in some applications to be able to handle \eqref{k-H1} when it has a non smooth $k$-convex solution. It has only been recently understood, c.f. \cite{Awanou-Std-fd-jsc} for the Monge-Amp\`ere equation, that what is needed is a numerical method provably convergent for smooth solutions and numerically robust to handle non smooth solutions. The approach in \cite{Awanou-Std-fd-jsc} is to regularize the data and use approximation by smooth functions. The key to numerically handle non smooth solutions of \eqref{k-H1} is to preserve $k$-convexity in the iterations.  For discrete $k$-convexity we simply require discrete analogues of the condition $S_l (D^2 u) \geq 0$ with a natural discretization of $D^2 u$. We refer to \cite{Aguilera2008} where this approach was first used for the discretization of $n$-convexity.
 
Consider the iterative method
\begin{align}
\begin{split}
 \Delta u^{m+1} & = \bigg( (\Delta u^{m})^k + \frac{1}{c(k,n)}(f-S_k (D^2 u^m)) \bigg)^{\frac{1}{k}} \, \text{in} \,  \Omega, 
  u^{m+1} = g \, \text{on} \, \partial \Omega, \label{k-H-iterative}
  \end{split}
\end{align}
with 
$c(k,n) = \binom{n}{k}/n^k$.

If $D^2 u$ has positive eigenvalues, we have the inequality
\begin{equation} \label{G-am}
S_k(D^2 u) \leq c(k,n) (\Delta u)^{k}, 
\end{equation}
which follows from the Maclaurin inequalities, 
\cite[Proposition 1.1 (v i)]{Gavitone2009}.

For $k=2$, \eqref{G-am} also holds with {\it no convexity assumption} on $u$, \cite[Lemma 15.11]{Lieberman96}. Explicitly 
$c(2,3)=1/3$. Also, $c(n,n)=1/n^n$ which gives
$$
\det D^2 u \leq \frac{1}{n^n} (\Delta u)^n, 
$$
a direct consequence of the arithmetic mean - geometric mean inequality.

If one starts with an initial guess $u^0$ such that $\Delta u^0 \geq 0$, \eqref{k-H-iterative} enforces $\Delta u^m \geq 0$ for all $m$. Indeed recall that $f \geq 0$ and assume that   $\Delta u^m \geq 0$. Then by \eqref{G-am}
$1/c(k,n) S_k (D^2 u^m) \leq (\Delta u^{m})^k$, and using \eqref{k-H-iterative} it follows that $(\Delta u^{m+1})^k \geq 0$. In other words, starting with an initial guess $u^0$ with $\Delta u^0 \geq 0$, \eqref{k-H-iterative} enforces subharmonicity 
in arbitrary dimension for smooth convex solutions and subharmonicity for 2-Hessian equations with no convexity assumption on $u$. In addition for 2-Hessian equations, the limit solution solves $S_2(D^2 u)=f \geq 0$. That is, the sequence $u^{m+1}$ defined by \eqref{k-H-iterative} has a formal limit which solves $\Delta u \geq 0$ and $S_2(D^2 u)\geq 0$. Thus \eqref{k-H-iterative} enforces 2-convexity in arbitrary dimension for 2-Hessian equations. 


 Another class of iterative methods we introduce 
 in this paper are Gauss-Seidel type iterative methods. 
 The Gauss-Seidel methods are more efficient than \eqref{k-H-iterative} for large scale problems.
 
 The simplicity of the methods discussed in this paper and the facility with which they can be implemented, make them attractive to researchers interested in Monge-Amp\`ere equations. The other major motivation to study the subharmonicity preserving iterations is that they can be adapted to the finite element context and have been numerically shown in that context to be robust for non smooth solutions.
 
 

In two dimension, \eqref{k-H-iterative} appears to perform well in the degenerate case $f \geq 0$ as discrete $k$-convexity is enforced in the iterations. The situation is different in three dimension with $k=2$.  
We were not able to reproduce the solution $u(x,y,z)=|x-1/2|$ by solving \eqref{k-H1} with $k=2$ and using \eqref{k-H-iterative}. Here, since $u$ does not depend on $z$, we have $f(x,y,z)=0$ as in the two dimensional case.
However, for $n=3$ and $k=3$, we can preserve convexity in the degenerate case by using the sequence of nonlinear $2$-Hessian equations
\begin{align} \label{sigma2k}
S_2(D^2 u^{m+1}) = 3 \bigg(\bigg(\frac{1}{3} S_2 (D^2 u^m)\bigg)^{\frac{3}{2}} + f - \det D^2 u^m
\bigg)^{\frac{2}{3}},
\end{align}
with $u^{m+1}=g$ on $\partial \Omega$. Each of these equations is solved iteratively by \eqref{k-H-iterative} with $k=2, n=3$. 
We note that $\bigg(\frac{1}{3} S_2 (D^2 u^m)\bigg)^{\frac{3}{2}} - \det D^2 u^m \geq 0$ when $S_2 (D^2 u^m)>0$, \cite[Lemma 15.12]{Lieberman96}. Starting with an initial guess which satisfies $S_2 (D^2 u^0) >0$ and setting $\det D^2 u^m=0$ in \eqref{sigma2k} whenever $S_2 (D^2 u^m)=0$, we obtain a double sequence iterative method which at the limit enforce $\Delta u \geq 0, S_2 (D^2 u) \geq 0$, and $ \det D^2 u = f \geq 0$. 

The reason for setting $\det D^2 u^m=0$ in \eqref{sigma2k} whenever $S_2 (D^2 u^m)=0$ is motivated by the observation that in the case $f=0$, if $S_2 (D^2 u^m)=0$, $S_2 (D^2 u^{m+1})$ is ill-defined or complex valued if $\det D^2 u^m>0$. While \eqref{k-H-iterative} may be inexact for degenerate 2-Hessian equations, its use inside a double iterative method appears effective. This is reminiscent of inexact Uzawa algorithms. 


\subsection{Relation with other work}
The $k$-Hessian equations have mainly applications in conformal geometry and physics. The Monge-Amp\`ere operator has received recently a lot of interest from numerical analysts. For $n=3$ and $k=2$, 
the numerical resolution of \eqref{k-H1} has been considered in \cite{Sorensen10}, where it was referred to as the $\sigma_2$ problem. 
The iterative method \eqref{k-H-iterative} generalizes an iterative method introduced in \cite{Benamou2010} for the two dimensional Monge-Amp\`ere equation. The latter corresponds to the choice $ k=n=2$ and the constant $c(2,2)=1/4$ replaced by 1/2. The $2$-Hessian equation has also been considered recently in \cite{FroeseObermanSalvago} from the point of view of monotone schemes.

 We will see that if the central finite difference discretization of \eqref{k-H-iterative} is solved by a Gauss-Seidel iterative method, one recovers a Gauss-Seidel iterative method which has been used by many authors to solve the two dimensional Monge-Amp\`ere equation. We will refer to the latter method as the 2D Gauss-Seidel method for Monge-Amp\`ere equation. It has been used in the numerical simulation of Ricci flow \cite{Headrick05}, as a smoother in multigrid methods for the balance vortex model in meteorology, \cite{Chen2010b,Chen2010c} and has been recently shown numerically to capture the viscosity solution of the 2D Monge-Amp\`ere equation \cite{Benamou2010}. The connection between \eqref{k-H-iterative} and the 2D Gauss-Seidel method for the Monge-Amp\`ere equation is what enables us to introduce new Gauss-Seidel type iterative methods for $k$-Hessian equations. 
 
 The ingredients of our proof of the convergence rate for the finite difference discretization are discrete Schauder estimates and a suitable generalization of the combined fixed point iterative method used in \cite{Feng2009}.
 Schauder estimates were also used in the proof of convergence of Newton's method at the continuous level \cite{Loeper2005}.

\subsection{Organization of the paper}
The paper is organized as follows: In the next section, we 
give some notations, recall the Schauder estimates and their discrete analogues. In section \ref{elliptic} we prove our main results on the quadratic convergence rate of a finite difference discretization of \eqref{k-H1} and in section \ref{newton-sec} we prove the convergence of Newton's method. In section \ref{convexity} we 
introduce new Gauss-Seidel type iterative methods and their connections with the subharmonicity preserving iterations \eqref{k-H-iterative}.
Section \ref{num} is devoted to numerical results. We conclude with some remarks. 
The reader interested only in the Monge-Amp\`ere equation, or for a first reading, may assume that $k=n$.

\section{Notation and preliminaries} \label{notation}

\subsection{H$\ddot{\text{o}}$lder spaces and Schauder estimates} \label{notation1}
For a nonnegative integer $r$ or for $r=\infty$, we denote by $C^r(\Omega)$ the set of all functions having all derivatives of order $\leq r$ continuous on $\Omega$ 
 and by $C^r(\tir{\Omega})$, the set of all functions in $C^r(\Omega)$
whose derivatives of order $\leq r$  have continuous
extensions to $\tir{\Omega}$. For a multi-index $\beta=(\beta_1,\ldots,\beta_n) \in \mathbb{N}^n$, put $|\beta|=\beta_1+\ldots+\beta_n$. We use the notation $D^{\beta} u(x)$ for the partial derivative
$(\partial /\partial x_1)^{\beta_1} \ldots (\partial /\partial x_n)^{\beta_n} u(x)$.

The norm in $C^r(\Omega)$ is given by
$$
||u||_{r;\Omega} = \sum_{j=0}^r \, |u|_{j;\Omega}, \quad |u|_{j;\Omega} = \text{sup}_{|\beta|=j} \text{sup}_{\Omega} |D^{\beta}u(x)|.
$$
We denote by $|x|$ the Euclidean norm of $x \in \R^n$. A function $u$ is said to be uniformly H$\ddot{\text{o}}$lder continuous with exponent $\alpha, 0 <\alpha \leq 1$ in $\Omega$ if
the quantity
$$
 \text{sup}_{x \neq y} \frac{|u(x)-u(y)|}{|x-y|^{\alpha}},
$$
is finite. 
The space $C^{r,\alpha}(\tir{\Omega})$
consists of functions whose $r$-th order derivatives are uniformly H$\ddot{\text{o}}$lder
continuous with exponent $\alpha$ in $\Omega$. It is a Banach space with norm
$$
||u||_{r,\alpha;\Omega} = ||u||_{r;\Omega} + [u]_{r,\alpha;\Omega},
$$
where
$$
[u]_{r,\alpha;\Omega}  =  \text{sup}_{|\beta|=r} \text{sup}_{x \neq y} \frac{|D^{\beta}u(x)-
D^{\beta} u(y)|}{|x-y|^{\alpha}}. 
$$
The norms $|| \, ||_{r;\Omega}$ and $|| \, ||_{r,\alpha;\Omega}$ are naturally extended to vector fields and matrix fields by taking the supremum over all components. We make the standard convention of using $C$ for a generic constant. For $A=(a_{ij})_{i,j=1,\ldots,n}$ and $B=(b_{ij})_{i,j=1,\ldots,n}$ 
we recall that 
$A:B=\sum_{i,j=1}^n a_{i j} b_{ij}$. 
We will often use the following property
\begin{equation} \label{alpha-prod1}
|| f g ||_{0,\alpha;\Omega} \leq C || f  ||_{0,\alpha;\Omega}  ||  g ||_{0,\alpha;\Omega}, \, \text{for} \, f,g \in C^{0,\alpha}(\tir{\Omega}),
\end{equation}
from which it follows that if $A, B$ are matrix fields 
\begin{equation}  \label{alpha-prod2}
||A:B||_{0,\alpha;\Omega} \leq C \sum_{i,j=1}^n || a_{ij} ||_{0,\alpha;\Omega} || b_{ij} ||_{0,\alpha;\Omega}.
\end{equation}

We first state a global regularity result for the solution of strictly elliptic equations, which follows from  \cite[Theorems 6.14, 6.6 and Corollary  3.8 ]{Gilbarg2001}.
\begin{thm} \label{SchauderPoisson}
Assume $0< \alpha < 1$. Let $\Omega$ be a $C^{2,\alpha}$ domain in $\R^n$ and $f, a^{ij} \in C^{\alpha}(\tir{\Omega})$, $\phi \in C^{2,\alpha}(\tir{\Omega})$. We consider the strictly elliptic operator
\begin{equation} \label{st-elliptic}
L u = \sum_{i,j=1}^n a^{ij}(x) \frac{\partial^2}{\partial x_i \partial x_j} u(x),
\end{equation}
with coefficients satisfying for positive constants $\lambda, \Lambda$, 
$$
\sum_{i,j=1}^na^{ij}(x) \zeta_i \zeta_j \geq \lambda \sum_{l=1}^n   \zeta_l^2, \zeta_l \in \R, \, \text{and} \, |a^{i,j}|_{0,\alpha;\Omega} \leq \Lambda.
$$
Then the solution $u$
of the equation
$$
L u =f \, \text{in} \, \Omega, u = \phi \, \text{on} \, \partial \Omega,
$$
satisfies
$$
||u||_{2,\alpha;\Omega} \leq C(||\phi||_{2,\alpha;\Omega}+ ||f||_{0,\alpha;\Omega}),
$$
where $C$ depends on $n, \alpha, \lambda, \Lambda, \Omega, \sup_{\partial \Omega} |\phi|$, and $\sup_{\Omega} |f|/\lambda$.
\smallskip
\end{thm}

We will make the slight abuse of language of also denoting by $S_k(x), x=(x_1,\ldots,x_n)$ the $k$th elementary symmetric polynomial of the variable 
$x$, i.e.
$$
S_k(\lambda) = \sum_{i_1 < \cdots < i_k} \lambda_{i_1}  \cdots \lambda_{i_k}.
$$
A function $u \in C^2(\Omega) \cap C^0(\tir{\Omega})$ with Hessian $D^2 u$ having eigenvalues $\lambda_i, i=1,\ldots,n$ is said to be $k$-admissible if
$S_j(\lambda) > 0, j=1,\ldots,k$. Solutions of the $k$-Hessian equation will be required to be $k$-admissible, thus requiring $f>0$.

Moreover, let $\kappa=(\kappa_1,\ldots,\kappa_{n-1})$ denote the principal curvatures of $\partial \Omega$. 

\begin{defn} \label{k-convexity-domain}
The domain $\Omega$ is said to be $(k-1)$-convex if there exists $c_0 >0$ such that
$$
S_{k-1}(\kappa) \geq c_0 >0 \, \text{on} \, \partial \Omega.
$$
\end{defn}
We then have, (\cite[Theorems 3.3 and 3.4 ]{WangXJ09}) 
\begin{thm} \label{k-Hessian}
Assume that $\Omega$ is $(k-1)$-convex, $\partial \Omega \in C^{3,1}$, $f \in C^{1,1}(\tir{\Omega})$, inf $f >0$, $g \in C^{3,1}(\tir{\Omega})$. Then there is a unique $k$-admissible solution $u \in C^{3,1}(\tir{\Omega})$ to the Dirichlet problem \eqref{k-H1}. 
\smallskip
\end{thm}

We will need some identities for the $k$-Hessian operator $S_k(D^2 u)$ which are derived explicitly for example in \cite[p. 5--6]{Gavitone2009}. See also \cite{WangXJ09}.
For a symmetric matrix $A=(a_{ij})_{i,j}=1,\ldots,n$ with eigenvalues $\lambda_i, i=1,\ldots,n$, let us also denote by $S_k(A)$ the $k$-th elementary symmetric polynomial of $\lambda$. This is equivalent to say that $S_k(A)$ is the sum of all $k \times k$ principal minors of $A$. Using the permutation definition of the determinant, we have
\begin{align} \label{k-minor}
S_k(A) = \frac{1}{k!}\sum_{1 \leq i_1,\cdots,i_k\leq n} \delta^{j_1,\cdots,j_k}_{i_1,\cdots,i_k} a_{i_1 j_1} \cdots a_{i_k j_k},
\end{align}
where $\delta^{j_1,\cdots,j_k}_{i_1,\cdots,i_k}$ is the generalized Kronecker delta which takes the value +1 if $i_1,\cdots,i_k$ differs from $j_1,\cdots,j_k$ by an even permutation and the value -1 in the case of an odd permutation. In other words, for a choice of $i_1,\ldots,i_k$, $\delta^{j_1,\cdots,j_k}_{i_1,\cdots,i_k}$ is the signature of the permutation $\sigma$ defined by $\sigma(i_l)=j_l, l=1,\ldots,k$. This implies that we only consider the case where the sets $\{i_1,\ldots,i_k\}$ and $\{j_1,\ldots,j_k\}$ are identical. Moreover we define $\delta^{j_1,\cdots,j_k}_{i_1,\cdots,i_k}$ to be 0 if $\{i_1,\ldots,i_k\} \neq \{j_1,\ldots,j_k\}$. Note also that $\{i_1,\ldots,i_k\}$ is a subset of $k$ elements of $\{ 1, \ldots, n \}$.

We have
\begin{align*}
S_k^{ij}(A)\coloneqq \frac{\partial}{\partial a_{ij}} S_k(A) = \frac{1}{(k-1)!} 
\sum_{1 \leq i, i_1,\cdots,i_{k-1}\leq n} \delta^{j,j_1,\cdots,j_{k-1}}_{i,i_1,\cdots,i_{k-1}} a_{i_1 j_1} \cdots a_{i_{k-1} j_{k-1}},  
\end{align*}
and so
$
S_k(A) = \frac{1}{k} \sum_{i,j=1}^n S_k^{ij}(A) a_{i j}
$
by the $k$-homogeneity of $S_k$ and Euler's theorem for homogeneous functions. Here $\{j_1,\ldots,j_{k-1}\}$ is the image of the set of $k-1$ elements $\{i_1,\ldots,i_{k-1}\}$ not containing $i$ by a permutation.

Let us denote by $\{S_k^{ij}(A) \}$ the symmetric matrix with entries $S_k^{ij}(A)$. We can write $S_k(A)=1/k \, \{S_k^{ij}(A) \}: A $, that is
$S_k(D^2 v) = \frac{1}{k} \{S_k^{ij}(D^2 v) \} : D^2 v$.
Using \eqref{k-minor} and observing that the expression of $S_k(A) $ can be written in terms of a multilinear map, we obtain
\begin{align} \label{k-Hdiv0}
S_k'(D^2 v) D^2 w =  \{S_k^{ij}(D^2 v) \}: D^2 w.
\end{align}

Let us denote by $ \{S_k^{ij}(A) \}' $ the Fr\'echet derivative of the mapping $A \to  \{S_k^{ij}(A) \}$. Since $\{S_k^{ij}(A) \}' (B)$ is a sum of terms each of which is a product of $k-2$ terms from $A$ and is linear in $B$, we have
\begin{equation} \label{Sij-der}
|| \{S_k^{ij}(D^2 v) \}' D^2 w  ||_{0;\Omega} \leq C |v|_{2;\Omega}^{k-2} |w|_{2;\Omega}.
\end{equation}
Using \eqref{alpha-prod2} and \eqref{Sij-der} we also have
\begin{equation} \label{cof-estimate}
|| \{S_k^{ij}(D^2 v) \}' D^2 w  ||_{0,\alpha;\Omega} \leq C |v|_{2,\alpha;\Omega}^{k-2} |w|_{2,\alpha;\Omega}.
\end{equation}

Finally we note that
\begin{lemma} \label{close}
Let $v$ be a $C^2$ strictly convex function with Hessian having smallest eigenvalue uniformly bounded below by a constant $a >0$. Then for $\eta=a/(2 n)$, we have $w$ strictly convex, whenever $||w-v||_{C^2(\Omega)} < \eta$. 
\end{lemma}
\begin{proof}
It follows from \cite[Theorem 1 and Remark 2 p. 39]{Hoffman53}  that for two symmetric $n \times n$ matrices $A$ and $B$, 
\begin{equation} \label{cont-eig}
|\lambda_l(A) - \lambda_l(B)| \leq n \max_{i,j} |A_{ij} - B_{ij}|, l=1, \ldots, n.
\end{equation}
It follows that for $u, v \in C^2(\Omega)$, 
\begin{align}
|\lambda_1( D^2 u(x)) - \lambda_1( D^2 v(x))| & \leq n ||w-v||_{C^2(\Omega)} \label{lambda1}.
\end{align}
The result then follows.
\end{proof}

We conclude this section with the equivalence of $n$-convexity and convexity in the usual sense.
\begin{lemma} \label{n-convexity}
A $C^2$ function $u$ is convex if and only if it is $n$-convex.
\end{lemma}

\begin{proof}
If $u$ is $C^2$, $\lambda_i \geq 0$ on $\Omega$ for all $i$ and thus $S_l(D^2 u) \geq 0, l=1,\ldots,n$. 

Conversely let us assume that $A$ is a symmetric matrix with
$S_l(A) \geq 0, l=1,\ldots,n$. We show that its eigenvalues $\lambda_i$ are all positive. Let
$$
p(\lambda)= \lambda^n + c_1 \lambda^{n-1} + \ldots + c_n,
$$
denote the characteristic polynomial of $A$. It can be shown \cite[Theorem 1.2.12]{Horn85} that
$$
c_l = (-1)^l S_l(A), l=1,\ldots,n.
$$
We show that if $\lambda_i <0$ then $p(\lambda_i) \neq 0$. We have
\begin{align*}
p(\lambda_i) & = \lambda_i^n + c_1 \lambda_i^{n-1} + \ldots + c_n \\
 & = \lambda_i^n + \sum_{l=1}^n (-1)^{l} S_l(A) \lambda_i^{n-l} \\
 & = (-1)^n \bigg( (-\lambda_i)^n + \sum_{l=1}^{n} (-1)^{l-n} S_l(A) \lambda_i^{n-l}\bigg) \\
 & = (-1)^n \bigg( (-\lambda_i)^n + \sum_{l=1}^{n}  S_l(A) (-\lambda_i)^{n-l}\bigg). 
\end{align*}
Since $-\lambda_i >0$ and $S_l(A) \geq 0$ for all $l$, we have $(-1)^n p(\lambda_i) \geq 0$. Moreover since 
$\sum_{l=1}^{n}  S_l(A) (-\lambda_i)^{n-l} \geq 0$ and $-\lambda_i > 0$ we have $(-1)^n p(\lambda_i) \neq 0$.
We conclude that $\lambda_i \geq 0$ for all $i$. This completes the proof.
\end{proof}

\subsection{Discrete Schauder estimates and related tools} \label{disc-schauder}
We will study the numerical approximation of \eqref{k-H1}--\eqref{k-H-iterative} by standard finite difference discretizations. 
For simplicity, we consider a cuboidal domain $\Omega = (0,1)^n \subset \R^n$. Let $0 < h < 1 \, \text{with} \, 1/h \in \mathbb{Z}$. Put
\begin{align*}
\mathbb{Z}_h & = \{x=(x_1,\ldots,x_n)^T \in \R^n: x_i/h \in \mathbb{Z} \}\\
\Omega^h_0 &= \Omega \cap  \mathbb{Z}_h, \Omega^h = \tir{\Omega} \cap  \mathbb{Z}_h, \partial \Omega^h = \partial \Omega \cap \mathbb{Z}_h= \Omega^h \setminus \Omega^h_0.
\end{align*}
Let $e^i, i=1,\ldots,n$ denote the $i$-th unit vector of $\mathbb{R}^n$. We define the following first order difference operators on the space $\mathcal{M}(\Omega^h)$ of grid functions $v^h(x), x \in \mathbb{Z}_h$,
\begin{align*}
\partial^i_{+} v^h(x) & \coloneqq \frac{v^h(x+he^i)-v^h(x)}{h}, \\
\partial^i_{-} v^h(x) & \coloneqq \frac{v^h(x)-v^h(x-he^i)} {h},\\
\partial^i_h v^h(x) & \coloneqq \frac{v^h(x+he^i)-v^h(x-he^i)}{2 h}.
\end{align*}
Higher order difference operators are obtained by combining the above difference operators. For a multi-index $\beta=(\beta_1,\ldots,\beta_n) \in \mathbb{N}^n$, we define
$$
\partial^{\beta}_{+} v^h \coloneqq  \partial^{\beta_1}_{+} \cdots \partial^{\beta_n}_{+}v^h.
$$
The operators $\partial^{\beta}_{-}$ and $\partial^{\beta}_{h}$ are defined similarly. Note that
\begin{align}
\partial^i_{+} \partial^i_{-} v^h(x) & = \frac{v^h(x+he^i)-2v^h(x)+v^h(x-he^i)}{h^2}, \label{second-disc1}
\end{align}
\begin{align}
\begin{split}
\partial^i_h \partial^j_h v^h(x) & = \frac{1}{4 h^2} \bigg\{v^h(x+he^i+h e^j)+v^h(x-he^i-h e^j) \\
& \qquad \qquad \qquad -v^h(x+he^i-h e^j)-v^h(x-he^i+ he^j)\bigg\}, i \neq j. \label{second-disc2}
\end{split}
\end{align}
The second order derivatives $\partial^2 v/\partial x_i \partial x_j$ are discretized using \eqref{second-disc1} and \eqref{second-disc2} for $i \neq j$. This gives a discretization of the Hessian $D^2 u$ which we denote by $\mathcal{H}_d(u^h)$.

Thus the discrete version of \eqref{k-H1} takes the form
\begin{align} \label{k-H1h}
S_k (\mathcal{H}_d \, u^h(x)) = f(x), x \in \Omega^h_0, u^h(x) = g(x)\, \text{on} \, \partial \Omega^h. 
\end{align}
The discrete Laplacian takes the form
\begin{align} \label{second-disc3}
\Delta_d (u^h) = \sum_{i=1}^n \partial^i_{+} \partial^i_{-} u^h.
\end{align}
We consider a discrete uniformly elliptic linear operator 
with low order terms
\begin{align*}
L_d v^h(x) = \sum_{i,j=1}^n a^{ij}(x) \partial^i_- \partial^j_+ v^h(x) + \sum_{i=1}^n b^{i}(x) \partial^i_+  v^h(x), x \in \Omega_0^h,
\end{align*}
i.e. the matrix $(a^{ij}(x))_{i,j=1,\ldots,n}$ is uniformly positive definite. 
We now define discrete analogues of the H$\ddot{\text{o}}$lder norms and semi-norms following \cite{Johnson74}. Let $[\xi,\eta]$ denote the set of points $\zeta \in \Omega^h$ such that $\xi_j \leq \zeta_j \leq \eta_j, j=1,\ldots,n$. Then for $v^h \in \mathcal{M}(\Omega^h), 0 < \alpha < 1$, we define
\begin{align*}
|v^h|_{j;\Omega_0^h} & = \, \text{max} \, \{\, |\partial^{\beta}_{+}v^h (\xi)|, |\beta|=j, [\xi,\xi+\beta] \subset \Omega^h \, \}, \\
[v^h]_{j,\alpha;\Omega_0^h} & = \, \text{max} \, \bigg\{\, \frac{|\partial^{\beta}_{+}v^h (\xi)- \partial^{\beta}_{+}v^h (\eta)|}{( |\xi-\eta|)^{\alpha}}, 
|\beta|=j, \xi \neq \eta, [\xi,\xi+\beta] \cup [\eta,\eta+\beta] \subset \Omega^h  \, \bigg\}, \\
||v^h||_{p;\Omega_0^h} & = \, \text{max}_{j \leq p} \, |v^h|_{j;\Omega_0^h},  \\
||v^h||_{p,\alpha;\Omega_0^h} & = ||v^h||_{p;\Omega_0^h}+ [v^h]_{p,\alpha;\Omega_0^h}. 
\end{align*}
The above norms are extended canonically to vector fields and matrix fields by taking the maximum over all components. For $j=0$, we have discrete analogues of the maximum and $C^{0,\alpha}$ norms. 

For a domain $O \subset \R^n$, we denote by $\mathcal{D}_h(O)$ the set of mesh functions on $\R^n$ which vanish outside $O$. If $v^h=0$ on $\partial \Omega^h$, extending $v^h$ by 0 to $\mathbb{Z}_h$, we obtain $v^h \in \mathcal{D}_h(\Omega)$. The following theorem then follows from 
\cite[Lemma 3.4]{Thomee1970}. 

\begin{thm} \label{discShauderPoisson}
Assume $ 0<\alpha<1$ and $v^h=0$ on $\partial \Omega^h$. Then there are constants $C$ and $h_0$ such that for $v^h \in \mathcal{M}(\Omega^h), h \leq h_0$
\begin{align} \label{discShauderPoisson0}
||v^h||_{2,\alpha;\Omega_0^h} \leq C ||L_d \, v^h||_{0,\alpha;\Omega_0^h}, 
\end{align}
with the constant $C$ independent of $h$.
\smallskip
\end{thm}

Since
\begin{align*}
\partial^i_{+} \partial^i_{-} v^h(x) & =\partial^i_{+} \partial^i_{+} v^h (x-h e^i)  \, \text{and} \, \\
\partial^j_{h} \partial^i_{h} v^h(x) & = \frac{1}{4}\bigg(\partial^j_{+} \partial^i_{+} v^h (x) + \partial^j_{+} \partial^i_{+} v^h (x- h e^i) +  \partial^j_{+} \partial^i_{+} v^h (x- h e^j) \\
& \qquad \qquad \qquad +\partial^j_{+} \partial^i_{+} v^h (x- h e^i-h e^j) \bigg), 
\end{align*}
we have max $\{||\partial^i_{+} \partial^i_{-} v^h ||_{0,\alpha;\Omega_0^h}, ||\partial^j_{h} \partial^i_{h} v^h||_{0,\alpha;\Omega_0^h}, i,j=1,\ldots,n  \} \leq  ||v^h||_{2,\alpha;\Omega_0^h}$ and hence the above theorem also applies when the second order derivatives \eqref{second-disc1}
and \eqref{second-disc2} are used in the definition of $|| . ||_{2,\alpha;\Omega_0^h}$.

By Taylor series expansions, it is not difficult to verify that for $v \in C^2(\Omega)$
\begin{align*}
|v |_{j;\Omega_0^h} \leq |v|_{2;\Omega}, j \leq 2.
\end{align*}
Moreover, for $v \in C^{4,\alpha}(\Omega)$,
\begin{align} \label{consistent}
||D^2 v - \mathcal{H}_d(v)||_{0;\Omega_0^h} \leq C h^2 |v|_{4;\Omega},
\end{align}
and 
\begin{align*} 
[D^2 v - \mathcal{H}_d( v) ]_{0,\alpha;\Omega_0^h} \leq C h^{2} [v]_{4,\alpha;\Omega}.
\end{align*}
To see that the last inequality holds, it is enough to consider a function of one variable $v \in C^{4,\alpha}(-1,1)$ and estimate 
$[v''(x)-(v(x+h)-2v(x)+v(x-h))/h^2]_{0,\alpha}$. Now,
$$
v''(x)-\frac{v(x+h)-2v(x)+v(x-h)}{h^2} = -\frac{h^2}{24} (v^{(4)}(x+ t_1 h) + v^{(4)}(x- t_2 h)), t_1, t_2 \in [0,1].
$$
Next we note that, using the definition, the $C^{0,\alpha}$ norm of $v^{(4)}(x+ t_1 h)$ is bounded above by the $C^{0,\alpha}$ norm of $v^{(4)}$. The result then follows.

We have for $v \in C^{4,\alpha}(\Omega)$,
\begin{align} \label{consistent2}
|| D^2 v - \mathcal{H}_d( v)||_{0,\alpha;\Omega_0^h} \leq C h^{2} ||v||_{4,\alpha;\Omega}.
\end{align}

\begin{lemma} \label{est-d2}
We have for $u \in C^{4,\alpha}(\Omega)$
$$
||S_k (D^2 u)-S_k (\mathcal{H}_d( u)) ||_{0,\alpha;\Omega_0^h}  \leq C h^{2} |u|_{2;\Omega}^{k-1} ||u||_{4,\alpha;\Omega}.
$$
\end{lemma}
\begin{proof}
By the mean value theorem, using \eqref{k-Hdiv0}, we have for some $t$ in $[0,1]$, and $x\in \Omega_0^h$, 
\begin{align*}
S_k (D^2 u) (x) - S_k (\mathcal{H}_d( u))(x) & = S_k'(t D^2(u)(x) + (1-t) \mathcal{H}_d( u)(x)):  (D^2 u (x) \\
& \qquad \qquad \qquad \qquad \qquad \qquad \qquad \qquad - \mathcal{H}_d( u)(x)) \\
& = \sum_{i,j=1}^n S_k^{ij} (t D^2(u)(x) + (1-t) \mathcal{H}_d( u)(x)) (D^2 u (x) \\
& \qquad \qquad \qquad \qquad \qquad \qquad \qquad \qquad - \mathcal{H}_d( u)(x))_{ij}.
\end{align*}
Using \eqref{alpha-prod2}, it follows that
\begin{align*}
||S_k (D^2 u)-S_k (\mathcal{H}_d( u) ) ||_{0,\alpha;\Omega_0^h} & \leq C (|u|_{2;\Omega} + | u|_{2;\Omega_0^h} )^{k-1} ||D^2 u - \mathcal{H}_d( u)||_{0,\alpha;\Omega_0^h} \\
& \leq C h^{2} |u|_{2;\Omega}^{k-1} ||u||_{4,\alpha;\Omega}. 
\end{align*}
\end{proof}

\section{Approximations by linear elliptic problems} \label{elliptic}
In this section, we prove the convergence of the iterative method \eqref{broyden} and its discrete version. As indicated in the introduction, we also obtain the existence and uniqueness of the solution of the discrete version of \eqref{k-H1}, i.e. \eqref {k-H1h}, as well as error estimates.

\subsection{Convergence at the operator level} 

We assume that 
there is a unique $k$-admissible solution $u \in C^{2,\alpha}(\tir{\Omega})$ of \eqref{k-H1} for $0 < \alpha < 1$. Let $u^0 \in C^{2,\alpha}(\tir{\Omega})$ such that 
\begin{equation} \label{is-delta}
||u-u^0||_{2,\alpha;\Omega} < \delta.
\end{equation}
For $k=n$, using an eigenvalue argument, it is not difficult to prove that the cofactor matrix is uniformly positive definite under the assumption $f \geq f_0 >0$ for a constant $f_0$. We assume that the matrix  $\{S_k^{ij}(D^2 u) \}$ is uniformly positive definite. We claim that this holds if $u \in C^2(\tir{\Omega})$ and there is $c_3 >0$ such that
 $$
 c_3 \leq S_l(D^2 u), 1 < l \leq k.
 $$
 We then have 
 \begin{equation} \label{c4}
 c_3 \leq S_l(D^2 u) \leq c_4, 1 < l \leq k,
 \end{equation}
  for a constant $c_4$. The proof is essentially given as \cite[Theorem 1.3 ]{Gavitone2009}. We define
 $$S_k^i(\lambda):= \frac{\partial}{\partial \lambda_i} S_k(\lambda).$$
 First we note from the proof of \cite[Theorem 1.3 ]{Gavitone2009} that the eigenvalues of $\{S_k^{ij}(D^2 u) \}$ are given by $S_k^i(\lambda(D^2 u)), 1 \leq i \leq n$. 
On the other hand, since $S_l(D^2 u) \geq c_3 >0, 1 < l \leq k$, we have by \cite[Proposition 1.1]{Caffarelli1985} 
$$
\frac{\partial}{\partial \lambda_i} S_k(\lambda)^{\frac{1}{k}} >0 \text{ for } \lambda=\lambda(D^2 u).
$$
Finally, as $S_k(D^2 u) \leq c_4$ and $u \in C^2(\tir{\Omega})$, the result follows.

By the continuity of the smallest eigenvalue of a matrix as a function of its entries,  $\{S_k^{ij}(D^2 u^0) \}$ is also uniformly positive definite for $|u-u^0|_{2;\Omega}$ sufficiently small. 

Next, $\{S_k^{ij}(D^2 u^0)\}$ is a symmetric matrix and divergence free by \cite[ Formula 1.10 ]{Gavitone2009}. Thus we obtain 
 \begin{equation} \label{div-prop}
 \div\bigg( \{S_k^{ij}(D^2 u^0) \} D v) \bigg) = \{S_k^{ij}(D^2 u^0) \}: D^2 v.
 \end{equation}

We have
\begin{thm} \label{contT-broyden}
Under the assumption that there is a unique $k$-admissible solution $u \in C^{2,\alpha}(\tir{\Omega})$ of \eqref{k-H1} for $0 < \alpha < 1$, the sequence defined by \eqref{broyden} converges to $u$
for $u^0$ sufficiently close to $u$. 
\end{thm}
\begin{proof}
We define the operator  $R: C^{2,\alpha}(\tir{\Omega}) \to C^{2,\alpha}(\tir{\Omega})$ by
\begin{align*}
-\div\bigg( \{S_k^{ij}(D^2 u^0) \} D  (v-R v)\bigg) &= - S_k (D^2 v) + f  \, \text{in} \, \Omega \\
R(v) & = g  \, \text{on} \, \partial \Omega.
\end{align*}
By Theorem \ref{SchauderPoisson}, the operator $R$ is well defined. We show that for $\rho>0$ sufficiently small, $R$ is a strict contraction in the ball
$B_{\rho}(u) = \{v \in C^{2,\alpha}(\tir{\Omega}), ||u-v||_{2,\alpha;\Omega} < \rho \}$. 

For $v, w \in B_{\rho}(u)$ we have using \eqref{div-prop}
\begin{multline*}
\div\bigg( \{S_k^{ij}(D^2 u^0) \} D  (R v-R w) \bigg) =\div\bigg(\{S_k^{ij}(D^2 u^0) \} D  ( v- w) \bigg) + S_k (D^2 w) - S_k (D^2 v)\\
 =  -\{S_k^{ij}(D^2 u^0) \} : (D^2 w - D^2 v)  + S_k (D^2 w) - S_k (D^2 v).
\end{multline*}
Next, by the mean value theorem and using \eqref{k-Hdiv0}, we have for some $t$ in $[0,1]$, 
\begin{multline*}
S_k(D^2 w)-S_k(D^2 v)  =  \{S_k^{ij}(t D^2 w + (1-t) D^2 v) \} : D^2 (w-v) \\
= \{S_k^{ij}(t (D^2 w -D^2 u^0)+ (1-t) (D^2 v-D^2 u^0) + D^2 u^0) \} : D^2 (w-v).
\end{multline*}

We use \eqref{cof-estimate} to estimate the $C^{0,\alpha}$ norm of 
$$
A=\{S_k^{ij}(t (D^2 w -D^2 u^0)+ (1-t) (D^2 v-D^2 u^0) + D^2 u^0) \}  - \{S_k^{ij}(D^2 u^0) \}.
$$
For $0 \leq s \leq 1$ to be specified below, put
$$
\alpha_{s t} = s t (D^2 w -D^2 u^0)+ s (1-t) (D^2 v-D^2 u^0) + D^2 u^0.
$$
We have
\begin{equation} \label{a-st}
|\alpha_{s t}|_{0,\alpha;\Omega} \leq ||u^0-v||_{2,\alpha;\Omega}+||u^0-w||_{2,\alpha;\Omega} +||u^0||_{2,\alpha;\Omega}.
\end{equation}
By the mean value theorem, for some $s \in [0,1]$ we have
$$
A=  \{S_k^{ij}(\alpha_{s t} ) \}' (t (D^2 w -D^2 u^0)+ (1-t) (D^2 v-D^2 u^0)),
$$
and thus by \eqref{cof-estimate}
\begin{equation} \label{a-st2}
||A||_{0,\alpha;\Omega} \leq C |\alpha_{s t}|_{0,\alpha;\Omega}^{k-2} (||u^0-v||_{2,\alpha;\Omega}+||u^0-w||_{2,\alpha;\Omega} ).
\end{equation}
By Schauder estimates (Theorem \ref{SchauderPoisson}), \eqref{alpha-prod2}, \eqref{a-st} and \eqref{a-st2} we obtain
\begin{align} \label{contraction-cont-level}
\begin{split}
||R(v) - R(w)||_{2,\alpha;\Omega}  & \leq C ||A||_{0,\alpha;\Omega}  ||D^2(v-w)||_{0,\alpha;\Omega} \\
& \leq C (||u^0-v||_{2,\alpha;\Omega}+||u^0-w||_{2,\alpha;\Omega} +||u^0||_{2,\alpha;\Omega} )^{k-2}  \\
& \qquad \qquad  (||u^0-v||_{2,\alpha;\Omega}+||u^0-w||_{2,\alpha;\Omega} ) ||v-w||_{2,\alpha;\Omega} \\
& \leq  C (\rho+\delta+||u^0||_{2,\alpha;\Omega})^{k-2} (\rho+\delta)||v-w||_{2,\alpha;\Omega},
\end{split}
\end{align}
where $\delta$ measures how close $u^0$ is to $u$ \eqref{is-delta}. 
Thus, for $\rho$ and $\delta$ sufficiently small,  $R$ is a strict contraction in $B_{\rho}(u)$. 

It remains to show that $R$ maps $B_{\rho}(u)$ into itself. We note by the definition of $R$ and unicity of the solution of \eqref{k-H1}, a fixed point of $R$ solves \eqref{k-H1}.  Let $v \in B_{\rho}(u)$,
\begin{align*}
||u-R v||_{2,\alpha;\Omega} & =||R u-R v||_{2,\alpha;\Omega} \leq ||u-v||_{2,\alpha;\Omega}  \leq  \rho,
\end{align*}
which shows that $R$ maps $B_{\rho}(u)$ into itself. The existence of a fixed point follows from the Banach fixed point theorem. Moreover, the sequence defined by $u^{m+1}=R(u^m)$, i.e. the sequence defined by 
 \eqref{broyden},  converges for $\rho$ and $\delta$ sufficiently small to $u$.
\end{proof}
\subsection{Finite difference discretization}

Next 
we consider the following discrete version of \eqref{broyden}
\begin{align} \label{broyden-D}
\begin{split}
  \{S_k^{ij}(\mathcal{H}_d  \, u^{0,h}) \} : \mathcal{H}_d u^{m+1,h} & =   \{S_k^{ij}(\mathcal{H}_d \, u^{0,h}) \} : \mathcal{H}_d u^{m,h}  \\
& \qquad \qquad \qquad  \qquad  +f-S_k (\mathcal{H}_d \, u^{m,h}) \, \text{in} \,  \Omega^h_0  \\
 u^{m+1,h} & = g \, \text{on} \, \partial \Omega^h.
 \end{split}
\end{align}
Under the assumptions of Theorem \ref{disc-thm} below, we show that \eqref{k-H1h} has a unique solution to which the above sequence converges. Moreover, the convergence rate is O($h^{2}$). 
Define
\begin{equation} \label{ball-h}
B_{\rho} ( u) = \{v^h \in \mathcal{M}(\Omega^h), ||v^h- u||_{2,\alpha;\Omega_0^h} \leq \rho \}.
\end{equation}
\begin{lemma} \label{sum-lem}
Let $S^h: \mathcal{M}(\Omega^h) \to \mathcal{M}(\Omega^h)$ be a strict contraction with contraction factor less than 1/2, i.e. for $v^h, w^h \in \mathcal{M}(\Omega^h)$
$$
||S^h(v^h) - S^h(w^h)||_{2,\alpha;\Omega_0^h} \leq\frac{1}{2} ||v^h - w^h||_{2,\alpha;\Omega_0^h}. 
$$
Let us also assume that $S^h$ does not move the center $u$ of the ball $B_{\rho} ( u)$ too far, i.e.
$$
||S^h( u) -  u||_{2,\alpha;\Omega_0^h} \leq C_0 h^2. 
$$
Then $S^h$ maps $B_{\rho} ( u)$ into itself for $\rho=2 C_0 h^{2}$. Moreover $S^h$ has a unique fixed point $u^h$ in $B_{\rho} ( u)$ with the error estimate
$$
|| u - u^h ||_{2,\alpha;\Omega_0^h} \leq 2 C_0 h^2.
$$ 
\end{lemma}
\begin{proof}
For $v^h \in B_{\rho} ( u)$,
\begin{align*}
||S^h(v^h)- u||_{2,\alpha;\Omega_0^h} & \leq ||S^h(v^h)- S^h( u)||_{2,\alpha;\Omega_0^h}+ ||S^h( u)- u||_{2,\alpha;\Omega_0^h}\\
& \leq \frac{1}{2} ||v^h- u||_{2,\alpha;\Omega_0^h} +  C_0 h^2 \\
& \leq \frac{\rho}{2}+C_0 h^2 \leq  \frac{\rho}{2} + \frac{\rho}{2}=\rho.
\end{align*}
This proves that  $S^h$ maps $B_{\rho} ( u)$ into itself. 
The existence of a fixed point follows from the Banach fixed point theorem. 
The convergence rate follows from the observation that
\begin{align*}
|| u - u^h ||_{2,\alpha;\Omega_0^h} & \leq || u - S^h( u) ||_{2,\alpha;\Omega_0^h} + ||S^h( u) - S^h(u^h) ||_{2,\alpha;\Omega_0^h} \\
 & \leq C_0 h^2  + \frac{1}{2} ||u^h- u||_{2,\alpha;\Omega_0^h}.
\end{align*}

\end{proof}

\begin{rem} \label{u0h-rem}
For $h$ sufficiently small, $\mathcal{H}_d (u)$ is sufficiently close to $D^2 u$ and hence $\{S_k^{ij}(\mathcal{H}_d  u)\}$ is positive definite, a property which also holds for $\{S_k^{ij}(\mathcal{H}_d \,u^{0,h})\}$ for $u^{0,h}$ sufficiently close to $u$. The arguments are similar to the ones of Lemma \ref{close}. See also Lemma \ref{lboundDelta} below.
\end{rem}
 
\begin{lemma}\label{lboundDelta} Let $u$ be a  $k$-admissible solution of \eqref{k-H1}. Assume that inf $f >0$ and $u \in C^4(\Omega)$. Then for $h$ sufficiently small, $\Delta_d (u) \geq c_0 >0$ where 
$c_0= 1/2 ((\text{inf} \, f)/c(k,n))^{1/k}$. Moreover, if $u$ is a strictly convex function, then for $\rho= O(h^{2})$, $\mathcal{H}_d( u)$ is a positive matrix and $v^h$ is a discrete convex function, when $v^h \in B_{\rho} ( u)$. 
\end{lemma} 
\begin{proof}
Since the eigenvalues of a matrix are continuous functions of its entries (as roots of the characteristic  polynomial), for a matrix $A=(a_{ij})$ with 
$S_k (A) >0$, we have for $\epsilon >0$, the existence of $\gamma >0$ depending only on the space dimension $n$ such that  $|S_k ( B) - S_k ( A)| < \epsilon$ when $\text{sup}_{ij} |b_{ij}-a_{ij}| <\gamma$. This implies $S_k ( B)> S_k ( A) - \epsilon$. Thus with $\epsilon=S_k( A)/2$, we have $S_k ( B) > S_k ( A)/2$. 

For $h$ sufficiently small we have $C h^2 |u|_{4;\Omega} < \gamma$ and thus since $S_k(D^2 u) =f >  \text{inf} \, f >0 $, by \eqref{consistent} $S_k(\mathcal{H}_d( u)) \geq 1/2 \, \text{inf} \, f $. 
By \eqref{G-am} 
$$
\Delta_d  (u) \geq \frac{1}{2} ((\text{inf} \, f)/c(k,n))^{1/k}.
$$

Let $v^h \in B_{\rho} ( u)$. Then by definition of $ B_{\rho} ( u)$ and \eqref{consistent}
\begin{align*}
||\mathcal{H}_d(v^h) - \mathcal{H}_d( u)||_{0,\alpha;\Omega_0^h} & \leq ||\mathcal{H}_d(v^h) - D^2 u||_{0,\alpha;\Omega_0^h} + ||D^2 u- \mathcal{H}_d( u)||_{0,\alpha;\Omega_0^h} \\
& \leq \rho + C h^2  |u|_{4;\Omega},
\end{align*}
which can be made smaller than $\gamma$ for $h$ and $\rho$ sufficiently small. Thus given that $\mathcal{H}_d( u)$ is positive definite, the same holds for $\mathcal{H}_d(v^h)$. 
\end{proof}

\begin{thm}\label{disc-thm}
Assume that $u \in C^{4,\alpha}(\tir{\Omega})$ is $k$-admissible. 
Choose $u^{0,h}$ such that 
$||u^{0,h}- u||_{2,\alpha;\Omega_0^h} = O(h^2)$.
For $h$ sufficiently small, \eqref{k-H1h} has a locally unique solution $u^h$ which satisfies $\Delta_d (u^h) \geq 0$ and 
$u^h$ converges to the unique solution $u$ of \eqref{k-H1} as $h \to 0$ at the rate O$(h^{2})$. 
\end{thm}

\begin{proof}
It follows from the assumptions that inf $f >0$. We define the operator
$R^h: \mathcal{M}(\Omega^h) \to \mathcal{M}(\Omega^h)$ by
\begin{align*} 
-  \{S_k^{ij}(\mathcal{H}_d \, u^{0,h}) \} : \mathcal{H}_d  (v^h-R^h v^h) &= - S_k (\mathcal{H}_d \, v^h) + f  \, \text{in} \, \Omega^h_0 \\
R^h(v^h) & = g  \, \text{on} \, \partial \Omega^h,
\end{align*}
and show that $R^h$ has a unique fixed point in $B_{\rho}( u)$ for $\rho=O(h^{2})$. By Remark \ref{u0h-rem} the above problem is then well defined. 
It follows from \eqref{div-prop} that the operator $R^h$ is a discrete version of the operator $R$ used in the proof of Theorem \ref{contT-broyden}. Thus, as in the proof of Theorem \ref{contT-broyden} 
we obtain 
\begin{align*}
  \{S_k^{ij}(\mathcal{H}_d \, u^{0,h}) \} : \mathcal{H}_d  (R^h v^h - R^h w^h) & =
S_k (\mathcal{H}_d \, w^h) - S_k (\mathcal{H}_d \, v^h) \\
& \qquad + \{S_k^{ij}(\mathcal{H}_d \, u^{0,h}) \}:\mathcal{H}_d \,(v^h-w^h). 
\end{align*}
And thus by the mean value theorem and discrete Schauder estimates, as in the proof of Theorem \ref{contT-broyden}
\begin{multline} \label{strict-step-01}
||R^h(v^h) - R^h(w^h)||_{2,\alpha;\Omega_0^h}   \leq \\ C (\rho+\delta_h+||u^{0,h}||_{2,\alpha;\Omega_0^h})^{k-2} (\rho+\delta_h)||v^h-w^h||_{2,\alpha;\Omega_0^h}.
\end{multline}
Next, note that with \eqref{consistent2} applied to $u$ one has $|u|_{2,\alpha;\Omega_0^h} \leq C ||u||_{4,\alpha;\Omega}$. 
It follows that $||u^{0,h}||_{2,\alpha;\Omega_0^h}\leq ||u^{}||_{2,\alpha;\Omega_0^h} + \delta_h \leq C ||u||_{4,\alpha;\Omega} +  \delta_h$. We recall that by assumption 
$||u^{0,h}- u||_{2,\alpha;\Omega_0^h} = O(h^2)$. 
Thus
$R^h$ is a strict contraction in $B_{\rho}( u)$ for $\rho=$O$(h^{2})$ and $h$ sufficiently small. 
Moreover, the contraction factor can be made smaller than 1/2 by choosing $h$ sufficiently small. 

Since $f=S_k(D^2 u)$, by the discrete Schauder estimates Theorem \ref{discShauderPoisson} and Lemma \ref{est-d2} 
\begin{align*}
||R^h( u) -  u||_{2,\alpha;\Omega_0^h} \leq C  ||S_k (D^2 u)-S_k (\mathcal{H}_d( u)) ||_{0,\alpha;\Omega_0^h} \leq C h^2. 
\end{align*}
By Lemma \ref{sum-lem} we conclude that $R^h$ has a fixed point $u^h$ in $B_{\rho}( u)$ with the claimed convergence rate.


The claimed property of $u^h$ follows from the fact that $u^h \in B_{\rho}( u)$ and Lemma \ref{lboundDelta}.

\end{proof}

\section{Newton's method} \label{newton-sec}

As in the previous section, we assume that $ \{S_k^{ij}(D^2 u) \}$ is uniformly positive definite. By Remark \ref{u0h-rem}, 
for $h$ sufficiently small, there exists $m' >0$ such that for $v^h \in B_{\rho}( u)$, $ \{S_k^{ij}(\mathcal{H}_d v^h) \}$ has smallest eigenvalue greater than $m'$.  We consider for $u^{0,h} \in B_{\rho}( u)$ the sequence of iterates
\begin{align} \label{newton}
\begin{split}
\{S_k^{ij}(\mathcal{H}_d u^{m,h}) \}: (\mathcal{H}_d u^{m+1,h} - \mathcal{H}_d u^{m,h}) & = f- S_k(\mathcal{H}_d u^{m,h}) \ \text{in} \ \Omega_0^h \\
 u^{m+1,h} & = g  \ \text{in} \ \partial \Omega^h.
\end{split}
\end{align} 
We note that \eqref{newton} defines $u^{m+1,h}$ as the solution of a discrete second order elliptic equation in non divergence form, which is uniformly elliptic for $u^{m,h} \in B_{\rho}( u)$ for $h$ sufficiently small. 

\begin{thm} \label{newton-th}
The sequence defined by \eqref{newton} satisfies 
\begin{equation} \label{newton-quad}
||u^{m+1,h} - u^h||_{2,\alpha;\Omega_0^h} \leq C ||u^{m,h} - u^h||_{2,\alpha;\Omega_0^h}^2,
\end{equation}
for $\rho$ and $h$ sufficiently small and where $u^h$ denotes the solution of \eqref{k-H1h} in $B_{\rho}( u), \rho=O(h^{2})$.  
\end{thm}

\begin{proof}
Put
\begin{equation} \label{B0}
B = \{S_k^{ij}(\mathcal{H}_d u^{m,h}) \}: (\mathcal{H}_d u^{m+1,h} - \mathcal{H}_d u^{h}) .
\end{equation}
We have by \eqref{k-H1h}
\begin{align} \label{B}
\begin{split}
B & = \{S_k^{ij}(\mathcal{H}_d u^{m,h}) \}: (\mathcal{H}_d u^{m,h} - \mathcal{H}_d u^{h}) +S_k (\mathcal{H}_d u^{h} )- S_k(\mathcal{H}_d u^{m,h}) \\
 & = \bigg( \{S_k^{ij}(\mathcal{H}_d u^{m,h}) \} - \{S_k^{ij}(\mathcal{H}_d u^{h} ) \} \bigg):  (\mathcal{H}_d u^{m,h} - \mathcal{H}_d u^{h})   \\
& \quad \quad + \{S_k^{ij}(\mathcal{H}_d u^{h}) \}: (\mathcal{H}_d u^{m,h} - \mathcal{H}_d u^{h}) +S_k (\mathcal{H}_d u^{h} )- S_k(\mathcal{H}_d u^{m,h}).
\end{split}
\end{align}
Put
\begin{equation} \label{B1}
B_1 = \bigg( \{S_k^{ij}(\mathcal{H}_d u^{m,h}) \} - \{S_k^{ij}(\mathcal{H}_d u^{h} ) \} \bigg):  (\mathcal{H}_d u^{m,h} - \mathcal{H}_d u^{h}),
\end{equation}
and
\begin{equation} \label{B2}
B_2 =  \{S_k^{ij}(\mathcal{H}_d u^{h}) \}: (\mathcal{H}_d u^{m,h} - \mathcal{H}_d u^{h}) +S_k (\mathcal{H}_d u^{h} )- S_k(\mathcal{H}_d u^{m,h}).
\end{equation}
By the mean value theorem, \eqref{k-Hdiv0} and \eqref{cof-estimate}, we have
$$
B_1 =\big( \{ S_k^{ij}(t \mathcal{H}_d u^{m,h} + (1-t) \mathcal{H}_d u^{h} ) \}'(\mathcal{H}_d u^{m,h} - \mathcal{H}_d u^{h}) \big):  (\mathcal{H}_d u^{m,h} - \mathcal{H}_d u^{h}),
$$
for $t \in [0,1]$ and thus
\begin{align} \label{B1-est}
\begin{split}
||B_1||_{0,\alpha;\Omega_0^h} & \leq C (||u^h||_{2,\alpha;\Omega_0^h} + ||u^{m,h}||_{2,\alpha;\Omega_0^h} )^{k-2} ||u^{m,h} - u^h||_{2,\alpha;\Omega_0^h}^2 \\
& \leq C (|| u||_{2,\alpha;\Omega_0^h} +\rho )^{k-2} ||u^{m,h} - u^h||_{2,\alpha;\Omega_0^h}^2 \\
& \leq C (|| u||_{2,\alpha;\Omega_0^h} +\rho )^{k-2} ||u^{m,h} - u^h||_{2,\alpha;\Omega_0^h}^2.
\end{split}
\end{align}
We also have by the mean value theorem 
\begin{align}  \label{B2-est}
\begin{split}
B_2 & = \{S_k^{ij}(\mathcal{H}_d u^{h}) \}: (\mathcal{H}_d u^{m,h} - \mathcal{H}_d u^{h}) \\
 & \qquad + \{S_k^{ij}( t  \mathcal{H}_d u^{h} + (1-t) \mathcal{H}_d u^{m,h}) \}: (\mathcal{H}_d u^{h} - \mathcal{H}_d u^{m,h})\\
 & = \bigg(\{S_k^{ij}(\mathcal{H}_d u^{h}) \} - \{S_k^{ij}( t  \mathcal{H}_d u^{h} + (1-t) \mathcal{H}_d u^{m,h}) \} \bigg) : (\mathcal{H}_d u^{m,h} - \mathcal{H}_d u^{h}) \\
 & =  \bigg( \{S_k^{ij}( (1-s)  \mathcal{H}_d u^{h}  +s t  \mathcal{H}_d u^{h} +  s(1-t) \mathcal{H}_d u^{m,h}) \}'  \\
 & \qquad \qquad \qquad \big((1-t)(\mathcal{H}_d u^{h} -  \mathcal{H}_d u^{m,h})\big) \bigg) : (\mathcal{H}_d u^{m,h} - \mathcal{H}_d u^{h}),
\end{split}
\end{align}
for $s,t \in [0,1]$. As for $B_1$ we obtain
\begin{align}
\begin{split}
||B_2||_{0,\alpha;\Omega_0^h} & \leq C (|| u||_{2,\alpha;\Omega_0^h} +\rho )^{k-2} ||u^{m,h} - u^h||_{2,\alpha;\Omega_0^h}^2.
\end{split}
\end{align}
Combining \eqref{B0}--\eqref{B2-est} 
and using Schauder estimates, we obtain \eqref{newton-quad}. 
\end{proof}

Choosing $\rho=O(h^{2})$ we have $C \rho < 1$ for $h$ sufficiently small. We conclude that $u^{m+1,h} \in B_{\rho}( u)$ when $u^{m,h} \in B_{\rho}( u)$ and the quadratic convergence rate of Newton's method.
\begin{rem}

Having established that the discrete problem has a locally unique solution and that $v^h$ is a discrete convex function for $v^h$ sufficiently close to $u$, the convergence of Newton's method also follows from the verification of standard assumptions given in \cite[p. 68]{Kelley95}. See \cite{Oberman2010a} for an example of verification of the standard assumptions for a wide stencil discretization.

\end{rem}

\section{Gauss-Seidel iterative methods} \label{convexity}

It is a natural idea to solve \eqref{k-H1h} by a nonlinear Gauss-Seidel method, that is solve \eqref{k-H1h} for $u^h(x)$ and solve the resulting nonlinear equations by a Gauss-Seidel method. Although this seems a daunting task for arbitrary $k$, we show that for $k=2$, this takes a very elegant form. We then establish a connection between  the resulting nonlinear Gauss-Seidel iterative method for $2$-Hessian equations and the discrete version of \eqref{k-H-iterative}, i.e.
\begin{align} \label{k-H1-iterativeD}
\begin{split}
\Delta_d \, u^{m+1,h}  & = \bigg( (\Delta_d \, u^{m,h})^k + \frac{1}{c(k,n)}( f-S_k (\mathcal{H}_d \, u^{m,h} ) )\bigg)^{\frac{1}{k}}  \ \text{in} \ \Omega_0^h \\
u^{m+1,h} & =  g \, \text{on} \, \partial \Omega^h,
\end{split}
\end{align}
when the Gauss-Seidel method is used to solve the Poisson equations. 

\subsection{Nonlinear Gauss-Seidel method for 2-Hessian equations}

We start with the identity
\begin{align} \label{identity}
\Delta_d \, u^{h}  = \bigg( (\Delta_d \, u^{h})^2 + \frac{1}{c(2,n)}(f-S_2 (\mathcal{H}_d \, u^{h} ) ) \bigg)^{\frac{1}{2}},
\end{align}
and show that the right hand side is independent of $u^h(x)$. Note that by \eqref{second-disc2}, $\partial^i_h \partial^j_h u^h(x), i \neq j$ is independent of $u^h(x)$ and 
by \eqref{second-disc3},
$$
\frac{\partial (\Delta_d \, u^{h}(x))}{\partial (u^h(x))} = \sum_{i=1}^n -\frac{2}{h^2} = -\frac{2 n}{h^2}.
$$
Since $\partial S_k(A)/\partial z = \sum_{i,j=1}^n( \partial S_k(A)/\partial a_{ij}) ( \partial a_{ij}/\partial z)$,
we conclude that
\begin{align*}
\frac{\partial }{\partial (u^h(x))} S_2 (\mathcal{H}_d \, u^{h}(x) ) & = \sum_{i,j=1 \atop i \neq j}^n  S_2^{ij}(\mathcal{H}_d \, u^{h}(x)) \frac{\partial}{\partial (u^h(x))} \partial^i_h \partial^j_h u^h(x) \\
& \qquad + \sum_{i=1}^n  S_2^{i i}(\mathcal{H}_d \, u^{h}(x)) \frac{\partial}{\partial (u^h(x))} \partial^i_+ \partial^i_- u^h(x) \\
& = -\frac{2}{h^2} \sum_{i=1}^n S_2^{ii}(\mathcal{H}_d \, u^{h}(x)) = -\frac{2}{h^2} \sum_{i=1}^n \sum_{1 \leq  p \leq n \atop p \neq i} \delta_{i p}^{i p} \,
\partial^p_+ \partial^p_- u^h(x)\\
& = -\frac{2}{h^2} \sum_{i=1}^n \sum_{p \neq i}  \partial^p_+ \partial^p_- u^h(x)=  -\frac{2}{h^2} (n-1) \Delta_d \, u^{h}(x)\\
& = -\frac{2}{h^2} (2 n) \,  c(2,n) \Delta_d \, u^{h}(x) = -\frac{4 n}{h^2} c(2,n) \Delta_d \, u^{h}(x),
\end{align*}
and we recall that the definition of $\delta_{i p}^{i p}$ was given in section \ref{notation1}.
This gives
$$
\frac{\partial }{\partial (u^h(x))} \bigg( (\Delta_d \, u^{h}(x))^2 + \frac{1}{c(2,n)}(f-S_2 (\mathcal{H}_d \, u^{h}(x) )) \bigg) = 0.
$$
We can therefore rewrite \eqref{identity} as 
\begin{align} \label{identity2}
\begin{split}
u^h(x) & = \frac{h^2}{2 n} \bigg[\sum_{i=1}^n \frac{u^h(x+he^i) + u^h(x-h e^i)}{h^2} \\
& \qquad \qquad \qquad \quad - \bigg( (\Delta_d \, u^{h}(x))^2 + \frac{1}{c(2,n)}(f-S_2 (\mathcal{H}_d \, u^{h}(x) ) \bigg)^{\frac{1}{2}} \bigg],
\end{split}
\end{align}
where the solution with $\Delta_d \, u^{h} \geq 0$ has been selected. For $n=2$, this is the identity which was solved in \cite{Headrick05,Chen2010b,Chen2010c,Benamou2010} by a Gauss-Seidel iterative method, as indicated in the introduction. For $n \geq 3$, this provides new iterative methods for the $2$-Hessian equations. 


Henceforth, we shall assume that a row ordering of the elements of $\Omega^h$ is chosen. Note that if we apply the Gauss-Seidel method to the problem \eqref{k-H1-iterativeD}, we obtain a double sequence $u^{m,p,h}$ defined by
\begin{align*} 
\begin{split}
u^{m+1,p+1,h}(x) & = \frac{h^2}{2 n} \bigg[\sum_{i=1}^n \frac{u^{m+1,p,h}(x+he^i) + u^{m+1,p+1,h}(x-h e^i)}{h^2} \\
& \qquad \qquad \qquad \quad - \bigg( (\Delta_d \, u^{m,h}(x))^2 + \frac{1}{c(2,n)}(f-S_2 (\mathcal{H}_d \, u^{m,h}(x) ) \bigg)^{\frac{1}{2}} \bigg],
\end{split}
\end{align*}
This leads us to consider the double sequence $u^{m,h}_p$ defined by
\begin{align*} 
\begin{split}
u^{m+1,h}_{p+1}(x) & = \frac{h^2}{2 n} \bigg[\sum_{i=1}^n \frac{u^{m+1,h}_p(x+he^i) + u^{m+1,h}_{p+1}(x-h e^i)}{h^2} \\
& \qquad \qquad \qquad \quad - \bigg( (\Delta_d \, u^{m,h}_{p*}(x))^2 + \frac{1}{c(2,n)}(f-S_2 (\mathcal{H}_d \, u^{m,h}_{p*}(x) ) \bigg)^{\frac{1}{2}} \bigg],
\end{split}
\end{align*} 
where $\Delta_d \, u^{m,h}_{p*}(x)$ and $S_2 (\mathcal{H}_d \, u^{m,h}_{p*}(x) )$ are the actions of the discrete Laplace and $2$-Hessian operators on $u^{m,h}_p$ updated with the most recently computed values.

Formally, as $m \to \infty$, this gives the nonlinear Gauss-Seidel method 
\begin{align} \label{Gauss} 
\begin{split}
u^{h}_{p+1}(x) & = \frac{h^2}{2 n} \bigg[\sum_{i=1}^n \frac{u^{h}_p(x+he^i) + u^{h}_{p+1}(x-h e^i)}{h^2} \\
& \qquad \qquad \qquad \quad - \bigg( (\Delta_d \, u^{h}_{p*}(x))^2 + \frac{1}{c(2,n)}(f-S_2 (\mathcal{H}_d \, u^{h}_{p*}(x) ) \bigg)^{\frac{1}{2}} \bigg],
\end{split}
\end{align}
where as above $\Delta_d \, u^{h}_{p*}(x)$ and $S_2 (\mathcal{H}_d \, u^{h}_{p*}(x) )$ are the actions of the discrete Laplace and $2$-Hessian operators on $u^{h}_p$ updated with the most recently computed values of $u^{h}_{p+1}$.
In particular, the right hand side of \eqref{Gauss} does not depend on $u^h_{p+1}$ since as shown above, the right hand side of \eqref{identity2} does not depend on $u^h(x)$.

\section{Numerical results} \label{num}
We give numerical results for the $\sigma_2$ problem, i.e. for $k=2, n=3$ using the subharmonicity preserving iterations. 
Although our theoretical results only cover smooth solutions, as indicated in the abstract and in the introduction, the subharmonicity preserving iterations appear able to handle non smooth solutions. The initial guess in all of our numerical experiments is taken as the finite difference approximation of the solution of the Poisson equation $\Delta u = 2 \sqrt{f}$ in $\Omega$ with $u=g$ on $\partial \Omega$.

We use the following test functions on the unit cube $[0,1]^3$:

Test 1: A smooth solution which is strictly convex, $u(x,y,z)=e^{x^2+y^2+z^2}$ so that 
$f(x,y,z)=4(3+x^2+y^2+z^2)e^{2(x^2+y^2+z^2)}$ and $g(x,y,z)=e^{x^2+y^2+z^2}$ on $\partial \Omega$.

\bigskip

Test 2: A smooth solution which is $2$-convex but not convex. It is known that for a radial function $u(x)=\phi(r), r=|x|, x \in \R^n$ the eigenvalues of $D^2 u$ are given by
$\lambda_1=\phi''(r)$ with multiplicity 1 and $\lambda_2=\phi'(r)/r$ with multiplicity $n-1$. See for example \cite[Lemma 2.1]{Felmer2003}. It follows that with $u(x,y,z)=\ln(a+x^2+y^2+z^2)$, we have $\phi(r)= \ln(a+r^2)$ and we get 
$\Delta u =  \frac{6 a + 2 r^2}{ (a+r^2)^2}  \geq 0, \, S_2(D^2 u)= 4\frac{3 a-r^2}{ (a+r^2)^3} \geq 0, \,  \det D^2 u = 2\frac{a-r^2} {(a+r^2)^2},$
in $[0,1]^3$. With $a=2$, $\det D^2 u$ takes negative values in $[0,1]^3$.

\bigskip

Test 3: A solution not in $H^2(\Omega)$, $u(x,y,z)=-\sqrt{3-x^2-y^2-z^2}$ 
so that $f(x,y,z)=-(x^2+y^2+z^2-9)/(-3+x^2+y^2+z^2)^2$ and $g(x,y,z)=-\sqrt{3-x^2-y^2-z^2}$ on $\partial \Omega$.

\bigskip

Test 4: No exact solution is known.  Here
$f(x,y,z)=1$ and $g(x,y,z)=0$. 

\bigskip

Test 5: A degenerate three dimensional Monge-Amp\`ere equation. We take $f(x,y,z)=0$ and $g(x,y,z)=|x-1/2|$. 
We use the double iterative method based on \eqref{sigma2k}. 

Numerically, the solution computed may not satisfy $S^2 D^2 u^{m} \geq 0$. At those points we set both $S_2 (D^2 u^{m})$ and $\det D^2 u^m$ to 0 in 
\eqref{sigma2k}. If the numerical value of $S_2(D^2 u^{m})$ is negative, then 0 is a better approximate value. Since $S_2( D^2 u^{m})$ is computed from $u^m$, the numerical value of $\det D^2 u^m$ would also be inaccurate. Since $u^m$ is expected to be an approximate solution of $u$ for which $\det D^2 u \geq 0$, a better approximation of $\det D^2 u^m$ at any stage where the latter is negative is also 0. It would be interesting to analyze the effect of these rounding off errors on the overall numerical convergence of the method. For example, one may analyze the convergence of the inexact double iteration. Similar situations appear with inexact Newton's methods and inexact Uzawa algorithms. 

The right hand side $f(x,y,z)$ can be computed from the exact solution $u(x,y,z)$ using the definition of $S_2(D^2 u)$ as the sum of the $2 \times 2$ principal minors.

For all tests but Test 3, we used the direct solver \eqref{k-H1-iterativeD}. For $h=2^6$, we run out of memory with \eqref{k-H1-iterativeD}. For Test 3, the Gauss-Seidel method was used since there is no memory issue for the latter with $h=2^6$. As expected, we have quadratic convergence (as $h \to 0$) for the smooth solutions of Tests 1 and 2 while enough data is not available to give the convergence rate for the singular solution of Test 3.

\begin{table} 
\begin{tabular}{c|ccccc}  
 \multicolumn{6}{c}{$h$}\\
 &    $1/2^1$ &  $1/2^2$ &  $1/2^3$&  $1/2^4$ &  $1/2^5$ \\
Error & 6.2328 $10^{-2}$ & 2.6556 $10^{-2}$ & 7.7836 $10^{-3}$ & 2.0616 $10^{-3}$ & 5.2449 $10^{-4}$  \\
Rate & &1.23 & 1.77&1.92 &1.97 
\end{tabular} 
\bigskip
\caption{Maximum error with Test 1.}
\end{table}

\begin{table} 
\begin{tabular}{c|ccccc}  
 \multicolumn{6}{c}{$h$}\\
 &    $1/2^1$ &  $1/2^2$ &  $1/2^3$&  $1/2^4$ &  $1/2^5$ \\
Error & 6.5241 $10^{-4}$ & 5.0653 $10^{-4}$ & 1.3850 $10^{-4}$ & 3.5587 $10^{-5}$ & 9.1276 $10^{-6}$  \\
Rate & &0.36 & 1.87&1.96 &1.96 
\end{tabular} 
\bigskip
\caption{Maximum error with Test 2.}
\end{table}

\begin{table} 
\begin{tabular}{c|ccc}  
 \multicolumn{4}{c}{$h$}\\
 &    $1/2^4$ &  $1/2^5$ &  $1/2^6$ \\
Error & 1.1084 $10^{-3}$ & 9.7971 $10^{-4}$ & 7.6618 $10^{-4}$ \\&& &  \\
Rate & &0.18 &0.35 \\
\end{tabular} 
\bigskip
\caption{Maximum error with Test 3.}
\end{table}

\begin{figure}[tbp]
\begin{center}
\includegraphics[angle=0, height=4.5cm]{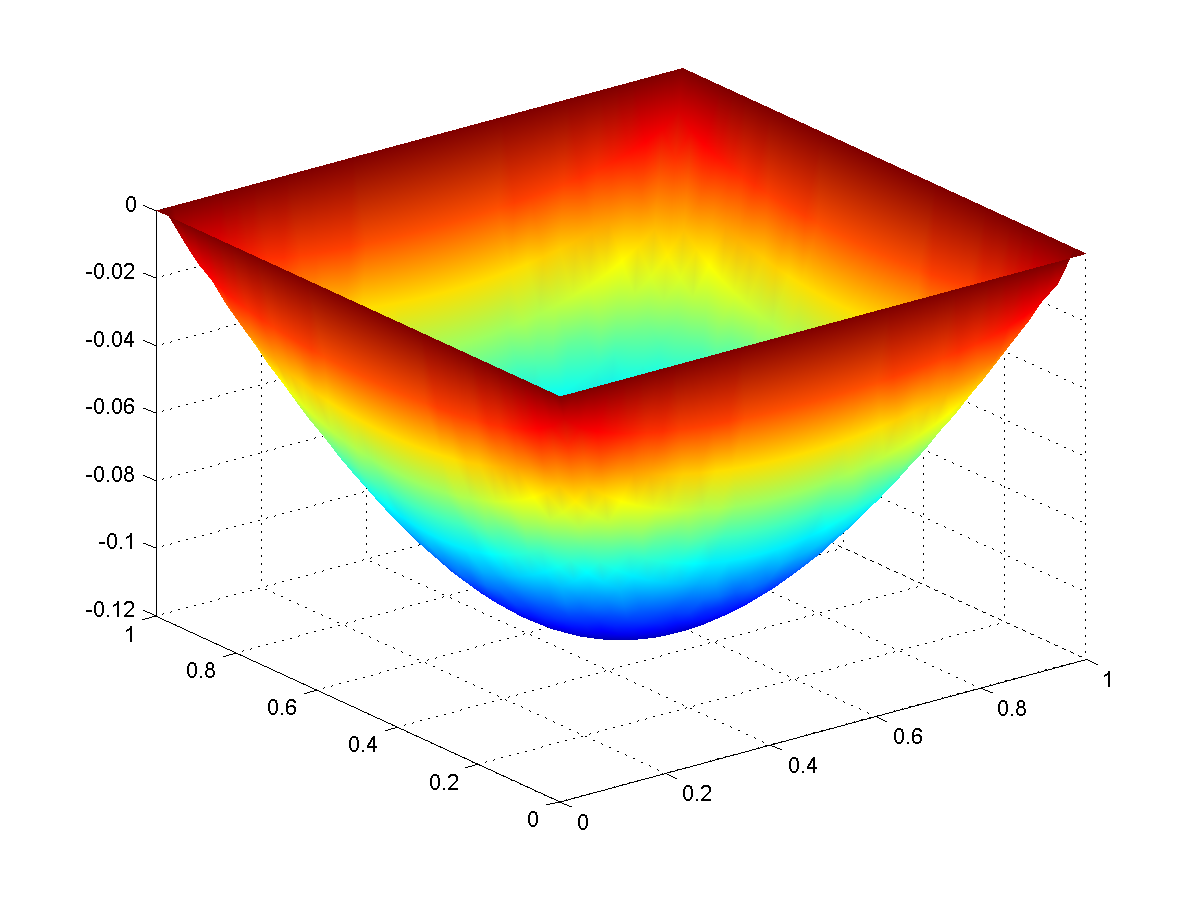}
\includegraphics[angle=0, height=5cm]{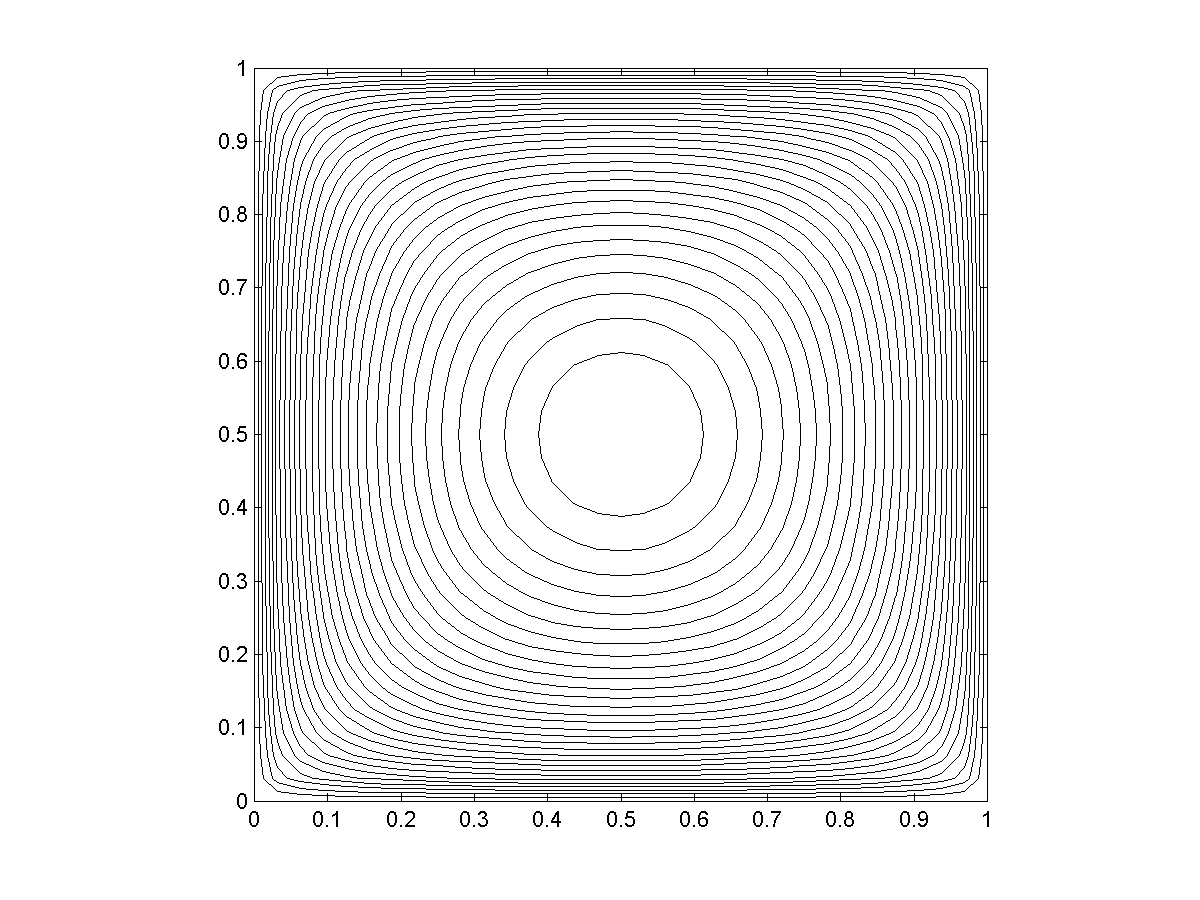}
\end{center}
\caption{Test 4, $h=1/2^5$. Graph and contour in plane $z=1/2$.}
\end{figure}

\begin{figure}[tbp]
\begin{center}
\includegraphics[angle=0, height=4.5cm]{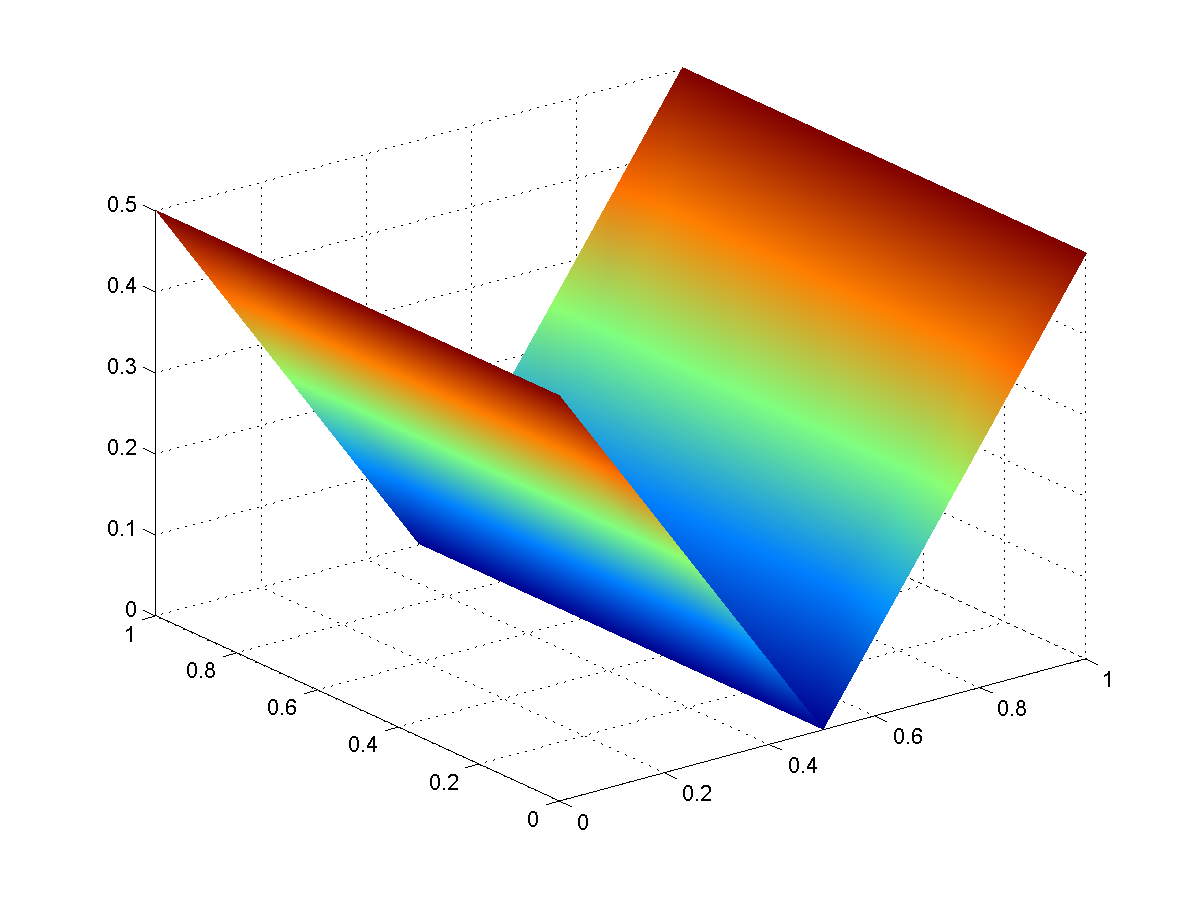}
\end{center}
\caption{Test 5, $h=1/2^4$. Graph in the plane $z=1/2$.}
\end{figure}

In 
\cite{Benamou2010}, it was argued based on numerical evidence that the Gauss-Seidel method \eqref{Gauss} is faster than a certain variant of the direct solver \eqref{k-H1-iterativeD} for singular solutions. In our implementation we saw evidence of the contrary, that is, the Gauss-Seidel method is less efficient. We note that the Gauss-Seidel method requires much more loops which are not efficient in MATLAB. 

\section{Concluding Remarks}

\begin{rem}
Although the pseudo-transient and time marching methods introduced in \cite{AwanouPseudo10} work as well for $k$-Hessian equations, and apply to more general fully nonlinear equations, the subharmonicity preserving iterative methods introduced in this paper are parameter free. 
All these type of methods can be accelerated with fast Poisson solvers and multigrid methods.
\end{rem}

\begin{rem}
When it comes to numerical methods for fully nonlinear equations, there are two types of convergence to study. Since the equations are nonlinear, they must be solved iteratively. One must then address the convergence to the discrete solution of the iterative methods used. The second type of convergence is the convergence of the numerical solution to the exact solution as the discretization parameter converges to 0. 
We have addressed both types of convergence in this paper. 
\end{rem}

\begin{rem}
Existence of a discrete solution and convergence (as the mesh size $h \to 0$), for finite difference discretization of smooth solutions of fully nonlinear equations, are not often discussed. 
It is clear that convergence does not simply follow from the consistency of standard finite difference discretization of the second order derivatives. For viscosity solutions, convergence of monotone, stable and consistent schemes follows immediately from the theory of Barles and Souganidis.
\end{rem}

\begin{rem}
The iterative method \eqref{k-H-iterative}  can be viewed as a linearization of the fully nonlinear equation \eqref{k-H1}. It is possible to linearize \eqref{k-H1} in ways different from \eqref{broyden} and \eqref{k-H-iterative}. See for example the methods described in \cite{AwanouPseudo10}. The iterative method \eqref{k-H-iterative} has been shown numerically to select discrete solutions which converge to non smooth solutions. Since \eqref{k-H-iterative} consists of a sequence of Poisson equations, the numerical solution of \eqref{k-H1} can now be tackled with any good numerical method. 
\end{rem}

{\bf Acknowledgments.}
The author would like to thank the referees for a careful reading of the manuscript. The author is grateful to M. Neilan for many useful discussions. The author was supported in part by NSF grants DMS-0811052, DMS-1319640 and the Sloan Foundation.

\end{document}